\documentclass[final,onefignum,onetabnum]{siamonline171218}

\usepackage{amssymb}
\usepackage{array}
\usepackage{xcolor}
\usepackage{hyperref}
\usepackage{amsfonts}
\usepackage{graphicx}
\usepackage{epstopdf}
\usepackage{lmodern}
\usepackage{algorithmic}
\ifpdf
  \DeclareGraphicsExtensions{.eps,.pdf,.png,.jpg}
\else
  \DeclareGraphicsExtensions{.eps}
\fi

\usepackage{enumitem}
\setlist[enumerate]{leftmargin=.5in}
\setlist[itemize]{leftmargin=.5in}

\newcommand{\dds}[1]{\frac{d #1}{d s}}
\newcommand{\dd}[2]{\frac{d #1}{d #2}}

\newcommand{\ip}[1]{\left\langle #1 \right\rangle} 
\newcommand{\avg}[1]{ #1 _{avg}} 
\newcommand{\R}{\mathbb{R}}

\newcommand{\mus}{{m_{us}}}
\newcommand{\integrate}{\int_{0}^{T}}
\newcommand{\summation}{\sum_{i=0}^{N-1}}
\newcommand{\dtaut}{D^t_\tau}
\newcommand{\dtot}{D^{t_2}_{t_1}} 
\newcommand{\adttau}{\overline{D}^\tau_t}
\newcommand{\ad}{\overline{D}}
\newcommand{\adto}{\overline{D}^{t_1}_{t_2}} 
\newcommand{\ap}{\overline{P}}
\newcommand{\al}{\overline{L}}
\newcommand{\az}{\overline{\zeta}}
\newcommand{\ac}{\overline{C}}
\newcommand{\av}{\overline{v}}
\newcommand{\aw}{\overline{w}}

\newcommand{\ay}{\overline{y}}

\newsiamremark{remark}{Remark}
\newsiamremark{hypothesis}{Hypothesis}
\crefname{hypothesis}{Hypothesis}{Hypotheses}
\newsiamthm{claim}{Claim}

\usepackage{amsopn}

\headers{Adjoint shadowing}{Angxiu Ni}
\title{Adjoint shadowing directions in hyperbolic systems for sensitivity analysis}

\author{Angxiu Ni\thanks{Department of Mathematics, University of California, Berkeley 
(\email{niangxiu@math.berkeley.edu}, \url{https://math.berkeley.edu/\~niangxiu/}).} }

\begin{document}
\maketitle

\begin{abstract}
  For hyperbolic diffeomorphisms, we define adjoint shadowing directions as a bounded inhomogeneous adjoint solution
  whose initial condition has zero component in the unstable adjoint direction.
  For hyperbolic flows, we define adjoint shadowing directions similarly, 
  with the additional requirement that the average of its inner-product with the trajectory direction is zero.
  In both cases, we show unique existence of adjoint shadowing directions, and how they can be used for adjoint sensitivity analysis.
  Our work set a theoretical foundation for efficient adjoint sensitivity methods for long-time-averaged objectives such as NILSAS.
\end{abstract}

\begin{keywords}
  adjoint, sensitivity analysis, linear response, dynamical systems, chaos, uniform hyperbolicity, ergodicity, shadowing lemma
\end{keywords}

\begin{AMS}
  34A34, 34D30, 37A99, 37D20, 37D45, 37N99, 46N99, 65P99
\end{AMS}

\section{Introduction}

Sensitivity analysis helps scientists and engineers design products \cite{Jameson1988,Reuther},
control processes and systems \cite{Bewley2001}, solve inverse problems \cite{Tromp}, 
estimate simulation errors \cite{Becker2001,Fidkowski,Persson2009_mesh_adapt}, 
assimilate measurement data \cite{Thepaut1991,COURTIER1993} 
and quantify uncertainties \cite{Marzouk2015}.
However, when the dynamical system is chaotic and the objective we are interested is a long-time-averaged quantity, 
conventional tangent or adjoint methods fail to provide useful sensitivity information.
One explanation of this failure is that trajectories of chaotic systems are highly sensitive to perturbations,
which property is typically mathematically modeled as the system being hyperbolic.

One approach to overcome the aforementioned difficulty is to move from investigating perturbations on trajectories
to perturbation of equilibrium distributions such as the SRB measures defined by Sinai, Ruelle, and Bowen \cite{young2002srb}.
This approach start systems from the same initial distribution and investigate its evolution after perturbing the parameters:
the new limit distribution yields sensitivity of averaged objectives we are interested in.
Such idea is reflected in Ruelle's linear response formula \cite{Ruelle_diff_maps,Ruelle_diff_maps_erratum,Ruelle_diff_flow},
and his fluctuation dissipation theorem \cite{Ruelle_newFDT}.
Ruelle's results were implemented by Lea \cite{Lea2000,eyink2004ruelle},
Abramov and Majda \cite{abramov2007blended,Abramov2008},
and Lucarini and others \cite{lucarini_linear_response_climate,lucarini_linear_response_climate2}.

Another approach is to keep analyzing perturbations in trajectories, but no longer insist on using the same initial conditions.
Instead, we look for a shadowing trajectory with perturbed parameters but still lies close to the base trajectory.
The existence of shadowing trajectories was first proved by Bowen \cite{Bowen_shadowing}, 
and Pilyugin \cite{Pilyugin_shadow_linear_formula} gave a formula of the first order difference between the shadowing trajectory and the base trajectory:
in this paper we call such first order difference the shadowing direction.
Wang developed the Least Squares Shadowing (LSS) method, \cite{wang2014convergence} 
where shadowing directions of hyperbolic diffeomorphisms are computed through a minimization of $L^2$ norms of inhomogeneous tangent solutions,
and the sensitivity is subsequently obtained.
For hyperbolic flows, a time dilation term was added
to reflect the speed difference between shadowing and the base trajectories \cite{Wang_ODE_LSS,Chater_convergence_LSS}.

Recently, the Non-Intrusive Least Squares Shadowing method (NILSS) developed by the author et al. \cite{Ni_NILSS_AIAA_2016,Ni_NILSS_JCP} 
finds a new formulation which allows constraining the minimization problem in LSS to the unstable subspace.
For many real-life problems, the unstable subspace has much lower dimension than the phase space, 
and NILSS can be thousands times faster than LSS. 
The Finite Difference NILSS (FD-NILSS) algorithm \cite{Ni_fdNILSS} can be implemented with only primal solvers, and does not require tangent solvers.
FD-NILSS has been applied to sensitivity analysis of several complicated flow problems 
\cite{Ni_CLV_cylinder,Ni_fdNILSS} which were too expensive for previous sensitivity analysis methods.

The marginal cost for a new parameter in NILSS is computing one extra inhomogeneous tangent solution.
Hence an adjoint algorithm is desired for cases where there are many parameters and only a few variables,
since cost of adjoint algorithms are not affected by the number of parameters.
The author et al. proposed an adjoint version in the first publication of NILSS \cite{Ni_NILSS_AIAA_2016};
however, this version is only correct for diffeomorphisms, 
for flows it lacks the constraint on the neutral adjoint CLV, which will be explained in our current paper.
Blonigan \cite{Blonigan_2017_adjoint_NILSS} developed a discrete adjoint version of NILSS, 
which requires both adjoint and tangent solvers: this can be a burden for programming \cite{Chandramoorthy2018_nilss_ad}.
To develop an shadowing-based adjoint sensitivity algorithm which does not require tangent solvers,
we should derive an analytic adjoint shadowing direction whose definition only depends on adjoint solutions:
this is our goal in this paper.

We organize the rest of this paper as follows.
We start by defining the adjoint shadowing direction and stating its unique existence theorem for both hyperbolic flows and diffeomorphisms;
then we review some properties of tangent and adjoint flows accompanying hyperbolic flows;
then we derive a formula which can be used in adjoint sensitivity analysis;
then we show this formula is in fact the adjoint shadowing direction for hyperbolic flows, and we prove its uniqueness.
Finally, we discuss adjoint shadowing direction for hyperbolic diffeomorphisms, which is easier than flows due to the absence of neutral subspace.
The appendices prove several properties of tangent and adjoint flows and how the two flows relate.
In another paper \cite{Ni_nilsas}, we develop the Non-Intrusive Least Squares Adjoint Shadowing (NILSAS) algorithm,
which is an efficient algorithm computing adjoint shadowing directions and performing adjoint sensitivity analysis.

\section{Statement of main theorems}\label{s:main thm}

In this section we first state the definition and the main theorem of the adjoint shadowing direction for hyperbolic flows.
Then we state the definition and the main theorem for hyperbolic diffeomorphisms.

\subsection{Adjoint shadowing in hyperbolic flows}

The governing equation for a uniform hyperbolic flow is: 
\begin{equation} \label{e:primal_system}
  \dd{u}{t} = f(u,s), \quad u(t=0) = u_0\,.
\end{equation}
We call a solution $u(t)$ a trajectory.
Here $f(u,s):\R^m\times \R\rightarrow\R^m$ is a smooth function, $u\in \R^m$ is the state, $u_0$ the initial condition, and $s\in\R$ is the parameter.
We assume there is only one parameter since, as we will see, our adjoint shadowing direction is not affected by perturbations on $s$.

We let a smooth function $J(u,s):\R^m\times\R\rightarrow\R$ be the instantaneous objective function, 
and the objective is obtained by averaging $J$ over a semi-infinite trajectory:
\begin{equation} \label{eq:average J}
  \avg{J}:= \lim\limits_{T\rightarrow\infty} \frac{1}{T}\integrate J(u,s)dt.
\end{equation}
For simplicity of discussions, we assume the system has a global attractor, hence $ \avg{J} $ only depends on $s$ but not on initial condition $u_0$.
We are interested in computing the sensitivity $d\avg{J} / d s$.
With above preparations we can now define the adjoint shadowing direction and then state the main theorem for hyperbolic flows.

\begin{definition}[adjoint shadowing for flows] \label{def:adjoint shadowing direction}
  On a trajectory $u(t), t\ge 0$ on the attractor, 
  the adjoint shadowing direction $\av: \R_+\rightarrow\R^m$ is defined as a function with the following properties:
  \begin{enumerate}
    \item $\av$ solves an inhomogeneous adjoint equation:
      \begin{equation} \label{e:av solve inhomo}
        \dd {\av}\tau + f_u^T \av = - J_u \,,
      \end{equation}
      where subscripts are partial derivatives, that is, $f_u = \partial f /\partial u$, $J_u = \partial J /\partial u$.
    \item $\av(t=0)$ has zero component in the unstable adjoint subspace. 
    \item $\|\av(t)\|$ is bounded by a constant for all $t\in\R_+$.
    \item The averaged inner-product of $\av$ and $f$ is zero:
      \begin{equation}
        \ip{\av, f}_{avg} := \lim_{T\rightarrow \infty} \frac 1T \int_0^T \ip{\av(t), f(t)} = 0 \,,
      \end{equation} 
      where $\ip{\cdot,\cdot}$ is the inner-product on Euclidean space.
  \end{enumerate}
\end{definition}
The unstable adjoint subspace will be defined in section~\eqref{s:adjoint flow}.
As we shall see, the last property of adjoint shadowing directions is mainly for uniqueness.
Then we state the main theorem of this paper for hyperbolic flows.

\begin{theorem}[adjoint shadowing for flows] \label{thm:main theorem}
  For a uniform hyperbolic dynamical system with a global compact attractor,
  on a trajectory on the attractor,
  there exists a unique adjoint shadowing direction.
  Further, we have the adjoint sensitivity formula:
  \begin{equation}\label{e:adjoint sensitivity}
    \dd {\avg{J}} {s} = \lim_{T\rightarrow\infty} \integrate \ip{\av, f_s} + J_s \, dt \,.
  \end{equation}
\end{theorem}
The definition of hyperbolic flows can be found in section~\ref{s:tangent flow}.

\subsection{Adjoint shadowing in hyperbolic diffeomorphisms}

The governing equation for a uniform hyperbolic diffeomorphism is: 
\begin{equation} \label{e:primal_system_diffeo}
  u_{i+1} = f(u_i,s), \quad i\ge 0\,.
\end{equation}
The objective is:
\begin{equation} \label{eq:average_J_diffeo}
  \avg{J}:= \lim\limits_{N\rightarrow\infty} \frac{1}{N} \sum_{i=0}^{N-1} J(u_i,s).
\end{equation}
Similar to flows, we assume $u_i\in \R^m$, $f(u,s)$ and $J(u,s)$ are smooth,
and there is only one parameter $s\in\R$.

\begin{definition}[adjoint shadowing for diffeomorphisms] \label{def:adjoint shadowing direction diffeomorphisms}
  On a trajectory $\{u_l\}_{l=0}^{\infty}$ on the attractor,
  The adjoint shadowing direction $\{\av_l\}_{l=0}^{\infty}$ is a sequence with the following properties:
  \begin{enumerate}
    \item $\{\av_l\}_{l=0}^{\infty}$ solves an inhomogeneous adjoint equation:
      \begin{equation} 
        \av_{l} = f_{ul}^T \av_{l+1} + J_{ul} \,, 
      \end{equation}
      where $f_{ul}:=$ $\partial f/ \partial u (u_l,s)$, and $J_{ul}:=$ $\partial J/ \partial u (u_l,s)$.
    \item $\av_0$ has zero component in the unstable adjoint subspace. 
    \item $\|\av_l\|$ is bounded by a constant for all $l\ge 0$.
  \end{enumerate}
\end{definition}

\begin{theorem}[adjoint shadowing for hyperbolic diffeomorphisms] \label{thm:main theorem for diffeomorphisms}
  For a uniform hyperbolic diffeomorphism with a global compact attractor,
  there exists a unique adjoint shadowing direction.
  Further, we have the adjoint sensitivity formula:
  \begin{equation}\label{e:djds_adjoint_diffeo}
    \dd {\avg{J}} {s} = \lim_{N\rightarrow \infty} \frac 1N \summation \left( \ip{\av_{l+1}, f_{sl}} + J_{si} \right) \,.
  \end{equation}
\end{theorem}

In the following text, we first spend several sections discussing adjoint shadowing in hyperbolic flows,
then we spend one section discussing adjoint shadowing in hyperbolic diffeomorphisms.

\section{Preparations}

For hyperbolic flows, constructing the formula of adjoint shadowing direction starts from tangent shadowing directions, 
the existence of which depends on the Covariant Lyapunov Vector (CLV) structure of the tangent flow.
Additionally, existence and uniqueness of the adjoint shadowing direction depends on the CLV structure of the adjoint flow.
In this section, as preparations, we study the tangent flow and adjoint flow of hyperbolic flows.

\subsection{The tangent flow} \label{s:tangent flow}

We begin by defining the homogeneous and inhomogeneous tangent equations and their solutions.
\begin{definition}[tangent equations]
  A homogeneous tangent solution $w(t):\R \rightarrow \R^m$ is a function which solves the homogeneous tangent equation:
  \begin{equation} \label{e:homogeneous tangent}
  \dd{w}{t} -f_u w =0 \,.
  \end{equation}
  An inhomogeneous tangent solution $v(t)$ is a function which solves:
  \begin{equation} \label{e:inhomogeneous tangent}
  \dd{v}{t} -f_u v = g(t) \,,
  \end{equation}
  where $g(t):\R \rightarrow \R^m$ is a vector-valued function of time.
\end{definition}

Tangent equations are also called variational equations.
Intuitively, homogeneous tangent solutions describe the evolution of perturbations on trajectories due to perturbations on initial conditions,
while inhomogeneous tangent solutions describe perturbations on trajectories due to perturbations on the system parameter $s$.
For homogeneous tangent equations, 
we define a propagation operator which maps the initial condition of a homogeneous tangent solution to its value at a later time.
This is the tangent flow operator.

\begin{definition}[tangent flow operator]
  The tangent flow operator $D_{t_1}^{t_2}$ is a linear operator on $\R^m$ whose action on any $w_1\in \R^m$ is given by:
  solving the homogeneous tangent solution $w(t)$ on time span $[t_1,t_2]$ with initial condition $w(t_1)=w_1$, then $D_{t_1}^{t_2} w_1 = w(t_2)$.
\end{definition}

Under this notation, for any $w_1\in \R^m$, $w(t) = D_{t_1}^{t}w_1$ is a homogeneous tangent solution.
The fact that $w(t)$ satisfies the homogeneous tangent equation can be written as:
\begin{equation}
\dd {}t \left(D_{t_1}^{t} w_1\right)= f_u(t) D_{t_1}^{t} w_1
\end{equation}
In this paper we use the tangent flow operator for writing down homogeneous tangent solutions with an given initial value.
In particular, we use this operator notation to define the Covariant Lyapunov Vectors (CLV), 
which are homogeneous tangent solutions whose Euclidean norm grows like exponential functions.
\begin{definition}[Lyapunov exponents and vectors]
  In this paper, a tangent CLV with Lyapunov Exponent (LE) $\lambda$ is a homogeneous tangent solution $\zeta(t)$ such that there is a constant $C$, 
  for any $t_1, t_2$, 
  \begin{equation} \label{e:CLV defn}
  \| \zeta(t_2) \| \le C e^{\lambda(t_2-t_1)}\| \zeta(t_1) \| \, . 
  \end{equation}
  Above inequality can also be written as:
  \begin{equation}
  \| D_{t_1}^{t_2} \zeta(t_1) \| \le C e^{\lambda(t_2-t_1)}\| \zeta(t_1) \| \, . 
  \end{equation}
\end{definition}

Notice that in both definitions for the flow operator and CLVs, it can happen either $t_1>t_2$, $t_1<t_2$, or $t_1=t_2$.
Moreover, by interchanging $t_1$ and $t_2$ in equation (\ref{e:CLV defn}), we get:
\begin{equation}
\| \zeta(t_1) \| \le C e^{\lambda(t_1-t_2)}\| \zeta(t_2) \| \, . 
\end{equation}
Together with equation (\ref{e:CLV defn}), we have:
\begin{equation}
\frac{1}{C} e^{\lambda(t_2-t_1)}\| \zeta(t_1) \| \le \| \zeta(t_2) \| \le C e^{\lambda(t_2-t_1)}\| \zeta(t_1) \| \, . 
\end{equation}
This shows that a CLV $\|\zeta(t)\|$ is bounded from both side by exponential functions with the same index but different coefficients.

We call the tangent CLVs with positive exponents unstable tangent CLVs, those with negative exponents stable, and the CLV with zero exponent the neutral CLV.
In this paper, we sort CLVs by the descending order of their corresponding LEs.
In this way, $\zeta_1$ is the fastest growing CLV, while $\zeta_m$ is the fastest decaying one.
We denote the number of unstable CLVs by $\mus$, where $us$ is short for `unstable'.
In this paper we assume the system is uniform-hyperbolic, which will give us more properties of CLVs.

\begin{definition}[uniform-hyperbolicity] \label{d:uniform hyperbolic}
  In this paper, a flow given by equation~\eqref{e:primal_system} is said to be uniform-hyperbolic if 
  there are $C\in(0,\infty)$ and $\lambda>0$, 
  such that for all $u$ on the attractor, there is a splitting of the tangent space $T_u=\R^m$ into stable, unstable, and neutral subspaces, that is,
  $T_u = V^+(u)\oplus V^-(u) \oplus  V^0(u)$, 
  such that on the trajectory passing $u$ at time $0$, we have:
  \begin{enumerate}
  \item for any $v\in V^+(u)$ and $t\le0$, $\|D_0^t v\|\le C e^{-\lambda |t|} \|v\|$;
  \item for any $v\in V^-(u)$ and $t\ge0$, $\|D_0^t v\|\le C e^{-\lambda |t|} \|v\|$;
  \item for any $v\in V^0(u)$, there is $a\in \R$ such that  $v= a f(u)$.
  \end{enumerate}
\end{definition}

We acknowledge that only a few dynamical systems are strictly uniform-hyperbolic; 
however, many dynamical systems in real-life approximately satisfy such assumption.
In fact, the uniform-hyperbolicity assumption is assumed in several important results, such as the shadowing lemma and the existence of SRB measure.
Moreover, shadowing-based algorithms like NILSS also assumes uniform-hyperbolicity, 
and it can give accurate sensitivities for real-life chaotic fluid mechanics problems \cite{Ni_CLV_cylinder,Ni_fdNILSS}.
Following these precedents, we also work under the uniform-hyperbolic assumption.

Tangent flow operators are cocycles and hence by the Oseledets theorem, there are $m$ CLVs.
By uniform-hyperbolicity, there is only one zero LE and absolute values of all the other LEs are greater than $\lambda$.
We can also see $V^+$, $V^-$, and $V^0$ in the definition of hyperbolicity are in fact span of unstable, stable, and the neutral CLVs, respectively. 
A look into some dynamical systems texts, such as \cite{Rezakhanlou2018}, tells us that directions of CLVs are continuous on the attractor.
Hence on a compact attractor we can find $\alpha>0$ such that at any time $t_0$, 
the angle between any CLV and the span of the rest of the CLVs is strictly larger than $\alpha$.
An immediate consequence is that at any time, the CLVs form a basis for $\R^m$, 
and hence we can define projection operators which project vectors onto this basis.

\begin{definition}[tangent projection operators]\label{def:tangent P}
  The projection operator onto the $j$-th CLV at time $t$, $P^j(t):\R^m\rightarrow\R^m$, is a linear operator such that
  \begin{equation}
    P^j(t) v := a_j \zeta _j(t) \,,
  \end{equation}
  where $v\in\R^m$, and $a_j$ is the $j$-th coordinate of $v$ under basis $\{\zeta_j(t)\}_{j=1}^m$,
  We define also the projection operators onto the unstable subspace, stable subspace, neutral subspace, and non-neutral subspace as:
  \begin{equation}
    P^+ := \sum_{j=1}^{\mus}P^j ;\;
    P^- := \sum_{j=\mus+2}^{n}P^j ;\;
    P^0 := P^{\mus+1} ;\;
    P^\pm := P^+ + P^- = I -P^0 \,.
  \end{equation}
\end{definition}

Lemma~\ref{l:matrix_form_of_P} in appendix \ref{app:properties_of_tangent_flow} shows how to write above projection operators in matrix form.
Additionally, lemma~\ref{l:DP=PD} shows that the tangent flow operator interchanges with the projection operator,
that is, for any $v$, $\tau$, $t$, and any projection operator $P$ defined above, we have $\dtaut P(\tau) = P(t) \dtaut$.

\subsection{The adjoint flow}\label{s:adjoint flow}

Adjoint CLVs are solutions of homogeneous adjoint equations whose norms grow exponentially.
The existence of adjoint CLVs are give by the Oseledets theorem, 
which is computationally verified by Kuptsov and Parlitz in \cite{Kuptsov2012_define_adjoint_CLV}.
However, for our purpose of deriving the adjoint shadowing direction, we want to understand in more details how adjoint CLVs relate to tangent CLVs:
these knowledge will help us construct the adjoint shadowing direction from its tangent counterpart.
This subsection investigates the CLV structure of the adjoint flow,
more specifically, we will state definitions and some properties of adjoint equations, adjoint flow operators, adjoint projection operators, and adjoint CLVs.

\begin{definition}[adjoint equations] \label{d:adjoint equations}
  A homogeneous adjoint solution $w(t):\R \rightarrow \R^m$ is a function which solves the homogeneous adjoint equation:
  \begin{equation} \label{e:homogeneous adjoint}
    \dd wt + f_u ^T w =0,
  \end{equation}
  where $\cdot^T$ is the transpose of a matrix.
  An inhomogeneous adjoint solution is one which solves:
  \begin{equation}\label{e:inhomogeneous adjoint}
    \dd {w}t + f_u ^T w =g(t),
  \end{equation}
  where $g(t):\R \rightarrow \R^m$ is a vector-valued function of time.
\end{definition}

In numerical implementations, we typically solve adjoint equations backwards in time.
This is because, as we will see, when solving backward in time, 
the dimension of the unstable adjoint subspace is the same as the unstable tangent subspace, which is typically much lower than $m$.
On the other hand, if we solve the adjoint equation forward in time, 
the unstable dimension will be much larger, causing strong numerical instability.

\begin{definition}[adjoint flow operator]
  The adjoint flow operator $\overline{D}^{t_1}_{t_2}:\R^m \rightarrow \R^m$ is a linear operator such that its action on any vector $w_2\in R^m$ is given by: 
  solve the homogeneous adjoint equation $w(t)$ on time span $[t_1, t_2]$ with terminal condition $w(t_2) = w_2$, 
  then $\ad^{t_1}_{t_2} w_2 = w(t_1)$.
\end{definition}

Notice that the definition holds for both $t_1<t_2$ and $t_1>t_2$. 
Under this notation,  $w(\tau) = \ad_{t_2}^{\tau}w_2$ is a homogeneous adjoint solution.
The fact that $w(\tau)$ satisfies the homogeneous adjoint equation can be written as:
\begin{equation} \label{e:diff_adjoint}
  \dd {}\tau \left(\ad_{t_2}^{\tau} w_2\right)= -f_u^T(\tau) \ad_{t_2}^{\tau} w_2
\end{equation}
Lemma~\ref{l:adjoint_pair_operators_products} shows that  $\ip{\ad^{t_1}_{t_2}w_2,v_1} = \ip{w_2, D^{t_2}_{t_1}v_1}$, 
this means $\ad^{t_1}_{t_2}$ is indeed the adjoint operator of $D^{t_2}_{t_1}$.

\begin{definition}[adjoint projection operators] \label{d:adjoint projection}
  At a given time $t$, the adjoint projection operator $\ap(t):\R^m \rightarrow \R^m$ is given by:
  \begin{equation}
    \ap(t) := P^T(t) \,,
  \end{equation}
  where $\cdot^T$ is the matrix transpose,
  and the projection operator $P$ can be either $P^j$, $P^+$, $P^-$, $P^0$, or $P^\pm$.
\end{definition}

Lemma~\ref{l:aP_P_orthogonal} shows that for $i\ne j$, the image space of the $P^i(t)$ is orthogonal to $\ap^j(t)$ for any $t$.
Lemma \ref{l:aDP=aPD} shows that similar to tangent case, adjoint flow operators commute with adjoint projection operators, that is,
$\adto \ap(t_2) = \ap(t_1)\adto$. 
Then we can use projection operators to define adjoint CLVs.

\begin{definition}[adjoint Lyapunov exponents and vectors]
  In this paper, an adjoint CLV with exponent $\lambda$ is a homogeneous adjoint solution $\az(t)$ such that there is a constant $C$, for any $t_1,t_2\in \R$, 
  \begin{equation}
  \| \adto \az (t_2)\|\le C e^{\lambda(t_2-t_1)} \|\az(t_2)\| \,.
  \end{equation}
  In particular, the neutral adjoint CLV, which is also denoted by $\ay$, is bounded by a constant independent of $t$.
\end{definition}

Notice that the time direction in the above definition is reversed: if the adjoint CLV grows exponentially backwards in time, its exponent is positive.
Lemma~\ref{l:ap_projects_to_adjoint_CLV} shows that the adjoint projection operator $\ap^j$ projects onto the adjoint CLV with exponent $\lambda_j$,
which is the exponent of the $j$-th tangent CLV:
this is a major result describing the relation between tangent and adjoint CLVs and not contained in previous literature. 
Another useful result is, assuming that we know the neutral adjoint CLV, $\ay$, 
lemma~\ref{l:ap0_from_ay} gives a formula for $\ap^0$: $\ap^0 v  = \ip{v,f} \ay / \ip{\ay, f} $.

\section{A candidate formula for adjoint shadowing directions}\label{s:candidate formula}

In this section we derive a formula for $\av$ which satisfies the adjoint sensitivity formula given in equation~\eqref{e:adjoint sensitivity}.
In the next section, we will show this formula is indeed the adjoint shadowing direction satisfying
definition~\ref{def:adjoint shadowing direction} and we will show its uniqueness.

While proving the convergence of the Least Squares Shadowing method for hyperbolic flows \cite{Chater_convergence_LSS},
Chater et. al. also proved that, under the same assumption of theorem \ref{thm:main theorem}, we have 
\begin{equation}\label{e:djds_lss}
  \dd {\avg{J}} {s} 
  = \lim_{T\rightarrow \infty} \frac 1T \integrate \left( \ip{J_u, v^\pm_T} + \eta \tilde{J} + J_s \right) dt \,,
\end{equation}
where $\tilde{J} = J - \avg{J}$.
Notice that $\avg{J}$ is averaged on an infinite trajectory, hence $\avg{J}$ and $\tilde{J}$ do not depend on $T$.
Now fix some $T$, and $v^\pm_T(t)$ is given by:
\begin{equation} \label{e:v(t)}
  v^\pm_T(t) = \int_0^t \dtaut P^-(\tau) f_s(\tau) d\tau - \int_t^T \dtaut P^+(\tau) f_s(\tau) d\tau \,.
\end{equation}
We can check that $v^\pm_T(t)$ is an inhomogeneous tangent which solves:
\begin{equation}
  \dd{v^\pm_T}{t} = f_u v^\pm_T + P^\pm f_s \,.
\end{equation}
We define $v^\pm(t) := \lim _{T\rightarrow \infty} v^\pm_T(t)$.
And $\eta$, whose definition is independent of $T$, is given by:
\begin{equation}
  \eta = - \frac{\ip{f, P^0 f_s}}{\ip{f,f}} \,.
\end{equation}
In this paper we call the pair of functions $(v^\pm, \eta)$ the tangent shadowing direction.
We start from sensitivity formula~\eqref{e:djds_lss} to derive a candidate for adjoint shadowing directions.

\subsection{Isolating dependency on parameters}

The major computational cost in numerical methods using equation~\eqref{e:djds_lss} to compute sensitivities, such as LSS and NILSS,
is to compute $v^\pm_T$ and the corresponding $\eta$, both of which depend on $f_s$, which in turn depends on the choice of parameter $s$.
Hence we need to recompute $v^\pm_T$ and $\eta$ for every new parameter.
In our adjoint formula, we want to isolate $f_s$,
that is, we want to transform the first and second term in equation~\eqref{e:djds_lss} into inner-products of $f_s$ with some other terms.
If this is achieved, we can develop algorithms such as NILSAS, whose computational cost does not scale with the number of parameters.

We first isolate $f_s$ in the second term in equation~\eqref{e:djds_lss}.
Using lemma~\ref{l:ap0_from_ay}, we have
\begin{equation}\label{e:adjoint eta}
\begin{split}
  \integrate \eta \tilde{J} dt
  &= \integrate  - \frac{\ip{f, P^0 f_s}}{\ip{f,f}} \tilde{J} 
  = \integrate  - \ip{\frac{ \tilde{J} }{\ip{f,f}} \ap^0 f , f_s} \\
  &= \integrate  - \ip{\frac{ \tilde{J} }{\ip{\ay,f}} \ay , f_s}
  = \ip{\av^0, f_s}_{L^2}  \,,
  \end{split}
\end{equation}
where $\ay$ is the neutral adjoint CLV, and $\av^0$, which is independent of $T$, is defined as:
\begin{equation}\label{e:define_av0}
  \av^0 := - \frac{ \tilde{J} }{\ip{f,f}} \ap^0 f =   - \frac{ \tilde{J} }{\ip{\ay,f}} \ay \,.
\end{equation}

Then we spend the next few paragraphs to isolate $f_s$ in the first term in equation~\eqref{e:djds_lss}.
In other words, we can view equation (\ref{e:v(t)}) as applying a linear operator $L^\pm_T$ on a function $f_s$ to obtain $v^\pm_T = L^\pm_T(f_s)$.
Now we want to find the adjoint operator $\al^\pm_T$ of $L^\pm_T$ such that:
\begin{equation}
  \ip{J_u, L^\pm_T(f_s)}_{L^2} = \ip{\al^\pm_T(J_u), f_s}_{L^2}
\end{equation}
where $\ip{\cdot,\cdot}_{L^2}$ denotes the inner product between two functions in the function space $L^2[0,T]$:
\begin{equation}
  \ip{f,g}_{L^2} = \integrate \ip{f(t),g(t)} dt \,.
\end{equation}

To obtain $\al^\pm_T$, we expand $L^\pm_T(f_s)$ and move the operations over to $J_u$. 
More specifically, using lemma~\ref{l:DP=PD} and the definition of $\adttau$ and $\ap^-(t)$, we have
\begin{equation} \label{e:adjoint L 1} \begin{split}
  &\ip{J_u(t), L^\pm_T(f_s)(t) }_{L^2}
  = \ip{J_u(t), v^\pm_T(t)}_{L^2} \\
  =& \int_0^T \ip{J_u(t), \int_0^t \dtaut P^-(\tau)f_s(\tau) d\tau}  dt - \int_0^T \ip{J_u(t), \int_t^T \dtaut P^+(\tau)f_s(\tau) d\tau}  dt \\
  =& \int_0^T \int_0^t \ip{J_u(t), \dtaut P^-(\tau)f_s(\tau) } d\tau dt - \int_0^T  \int_t^T  \ip{J_u(t),\dtaut P^+(\tau)f_s(\tau) } d\tau dt \\
  =& \int_0^T \int_0^t \ip{J_u(t), P^-(t) \dtaut f_s(\tau) } d\tau dt - \int_0^T  \int_t^T  \ip{J_u(t), P^+(t) \dtaut f_s(\tau) } d\tau dt \\
  =& \int_0^T \int_0^t \ip{\adttau \ap^-(t) J_u(t), f_s(\tau) } d\tau dt - \int_0^T  \int_t^T  \ip{\adttau \ap^+(t) J_u(t), f_s(\tau) } d\tau dt \,,
\end{split} \end{equation}

Now we can change the integration order in equation~\eqref{e:adjoint L 1} to get:
\begin{equation} \label{e:adjoint L 2} \begin{split}
  &\ip{J_u(t), L^\pm_T(f_s)(t) }_{L^2}
  = \ip{J_u(t), v^\pm_T(t)}_{L^2} \\
  =& \int_0^T \int_\tau^T \ip{\adttau \ap^-(t) J_u(t), f_s(\tau) } dt d\tau  - \int_0^T  \int_0^\tau  \ip{\adttau \ap^+(t) J_u(t), f_s(\tau) }  dt d\tau \\
  =&\int_0^T \ip{ \Bigl( \int_\tau^T \adttau \ap^-(t) J_u(t) dt - \int_0^\tau \adttau \ap^+(t) J_u(t) dt \Bigr) , f_s(\tau) } d\tau \\
  =& \ip{ \al^\pm_T(J_u)(\tau) , f_s(\tau) }_{L^2} 
  =\ip{ \av^\pm_T(\tau), f_s(\tau) }_{L^2} \,,
\end{split} \end{equation}
where the function $\av^\pm_T$ and operator $\al^\pm_T$ are defined as:
\begin{equation} \label{e:define_avpm}
  \av^\pm_T(\tau) = \al^\pm_T(J_u)(\tau) 
  :=  \int_\tau^T \adttau \ap^-(t) J_u(t) dt - \int_0^\tau \adttau \ap^+(t) J_u(t) dt \, .
\end{equation}
Hence, $\al^\pm_T$ is the adjoint operator of $L^\pm_T$.

With equation~\eqref{e:adjoint eta} and equation~\eqref{e:adjoint L 2}, we transform the right side of equation~\eqref{e:djds_lss} to: 
\begin{equation}
  \integrate \left( \ip{J_u, v^\pm_T} + \eta \tilde{J} + J_s \right) dt = \integrate \left( \ip{\av_T, f_s} + J_s \right) d\tau\,.
\end{equation}
with $\av_T$  defined as  $\av_T :=  \av^\pm_T + \av^0$.
Now the sensitivity formula in equation~\eqref{e:djds_lss} becomes:
\begin{equation}\label{e:djds_adjoint_1}
  \dds{\avg{J}} = \lim_{T\rightarrow\infty} \frac 1T \integrate \left(\ip{\av_T, f_s} + J_s \right) d\tau\,.
\end{equation}
We have achieved our goal to isolate the dependency on $f_s$ in the expression for sensitivities.

\subsection{Extending to semi-infinite trajectories}\label{s:extend to semi infinite}

Notice that $\av^\pm_T$ and $\av_T$ still depend on $T$, but adjoint shadowing solutions are defined on a semi-infinite trajectory, so we define 
\begin{equation} \label{e:define_avpm_semiinfi}
  \av^\pm (\tau):= \lim_{T\rightarrow\infty} \av^\pm_T (\tau)
  = \int_\tau^\infty \adttau \ap^-(t) J_u(t) dt - \int_0^\tau \adttau \ap^+(t) J_u(t) dt  \,;
\end{equation}
\begin{equation} \label{e:define_av_semiinfi}
  \av (\tau):= \lim_{T\rightarrow\infty} \av_T(\tau) = \av^\pm (\tau)+ \av^0 (\tau)\,.
\end{equation}
$\av$ is our candidate for the adjoint shadowing direction.
First of all, for $\av$ to be well-defined, we show that the limit in equation~\eqref{e:define_avpm_semiinfi} exists point-wise.

\begin{lemma}
  For any $\tau$, $\lim_{T\rightarrow\infty} \av_T (\tau)$ exists.
\end{lemma}

\begin{proof}
  We first fix an arbitrary $\tau\ge 0$.
  Since $\av^0$ does not depend on $T$, we only need to show $\lim_{T\rightarrow\infty} \av^\pm_T (\tau)$ exists.
  We just need to show that $\av^\pm_{T_2}(\tau) - \av^\pm_{T_1}(\tau)\rightarrow 0$ as $T_1, T_2 \rightarrow \infty$.
  We can assume $T_2 \ge T_1 \ge \tau$, and in view of equation~\eqref{e:define_avpm},
  \begin{equation} \begin{split}
    &\| \av^\pm_{T_2}(\tau) - \av^\pm_{T_1}(\tau) \|
    = \left\| \int_{T_1}^{T_2} \adttau \ap^-(t) J_u(t) dt \right\|
    \le \int_{T_1}^{T_2} \left\|\adttau \ap^-(t) J_u(t) \right\| dt \,.
  \end{split} \end{equation}
  By our hyperbolic assumption, stable adjoint CLVs have exponents smaller than $-\lambda$. 
  Now by lemma~\ref{l:bound for ap}, and that a continuous $J_u$ is bounded on a compact attractor, we have
  \begin{equation} \begin{split} \label{e:estimate}
    &\| \av^\pm_{T_2}(\tau) - \av^\pm_{T_1}(\tau) \|
    \le \int_{T_1}^{T_2} e^{-\lambda(t-\tau)} \left\|\ap^-(t) J_u(t) \right\| dt \\
    \le& C_\alpha  \int_{T_1}^{T_2} e^{-\lambda(t-\tau)}\|J_u(t)\| dt 
    \le C_\alpha \|J_u(t)\|_{L^\infty}\int_{T_1}^{T_2} e^{-\lambda(t-\tau)} dt \\
    =& \frac 1\lambda C_\alpha \|J_u(t)\|_{L^\infty} \left[e^{-\lambda(T_1-\tau)}- e^{-\lambda(T_2-\tau)}\right] \rightarrow 0, 
    \textnormal{ as } T_1, T_2\rightarrow \infty \,.
  \end{split} \end{equation}
\end{proof}

Notice that there is slight difference between the adjoint sensitivity formula
in  equation~\eqref{e:adjoint sensitivity} and equation~\eqref{e:djds_adjoint_1}.
That is, equation~\eqref{e:adjoint sensitivity} has first the $T$ in $\av_T$ goes to infinity, 
then the $T$ in the integration goes to infinity,
whereas in equation~\eqref{e:djds_adjoint_1} the two limit process happen at the same time.
To show equivalence between the two formula, we prove the following lemma.

\begin{lemma}We have:
  \begin{equation} 
    \lim_{T\rightarrow\infty} \frac 1T \integrate \ip{\av_T, f_s}  d\tau
    = \lim_{T\rightarrow\infty} \frac 1T \integrate \ip{\av, f_s}  d\tau \,.  
  \end{equation}
\end{lemma}

\begin{proof}
  It is equivalent to show that 
  \begin{equation}
    \lim_{T\rightarrow\infty} \frac 1T \integrate \ip{\av - \av_T, f_s}  d\tau
    = \lim_{T\rightarrow\infty} \frac 1T \integrate \ip{\av^\pm - \av^\pm_T, f_s}  d\tau = 0\,.
  \end{equation}
  Let $T_2\rightarrow\infty$ in equation~\eqref{e:estimate}, and using that a continuous $f_s$ is bounded on a compact attractor, we have
  \begin{equation} \begin{split}
    &\left| \frac 1T \integrate \ip{\av^\pm - \av^\pm_T, f_s}  d\tau \right|
    \le \frac 1T \integrate \left\| \av^\pm - \av^\pm_T \right\| \|f_s\|  d\tau \\
    \le & \frac 1{\lambda }\ C_\alpha \|J_u(t)\|_{L^\infty} \|f_s\|_{L^\infty} \frac 1T \integrate e^{-\lambda(T-\tau)}  d\tau \\
    =& \frac 1{\lambda^2 }\ C_\alpha \|J_u(t)\|_{L^\infty} \|f_s\|_{L^\infty} \frac 1T(1-e^{-\lambda T}) 
    \rightarrow 0 \,.
  \end{split} \end{equation}
\end{proof}

As a result we can change $\av_T$ to $\av$ in equation~\eqref{e:djds_adjoint_1},
and show that our $\av$ defined in this section can be used for adjoint sensitivity analysis as in equation~\eqref{e:adjoint sensitivity}.
The next question is to show that $\av$ is the unique adjoint shadowing direction as in definition~\ref{def:adjoint shadowing direction}.

\section{Existence and uniqueness of adjoint shadowing directions}

When designing algorithms, we will not use the definition of $\av$ in equation~\eqref{e:define_av_semiinfi}, 
since its expression involves two seemingly unrelated parts, both having complicated expressions.
Instead, while developing algorithms like NILSAS, we reverse engineer.
That is, we generate a function and check if it has the same properties as the adjoint shadowing direction.
Hence a list of unified and simple properties which uniquely determines the adjoint shadowing direction is important for algorithm design.
In this section, we first check that our candidate, $\av$, satisfies the four properties listed in definition~\ref{def:adjoint shadowing direction},
thus showing the existence of adjoint shadowing directions.
Then, we show uniqueness.

\subsection{Proving the first property}

It is straight forward to derive a candidate formula as we did,
but it is not obvious the summation of two parts, $\av^\pm$ and $\av^0$, should have a unified relation as given in the first property.
In this subsection,
we prove the first property by showing that both $\av^\pm$ and $\av^0$ solve inhomogeneous adjoint equations, the sum of whose right-hand-sides is $-J_u$.

\begin{lemma} \label{l:dvpm dt}
  $\av^\pm$ defined in equation~\eqref{e:define_avpm_semiinfi} solves the inhomogeneous adjoint equation:
\begin{equation}
  \dd {\av^\pm}\tau + f_u^T \av^\pm = - \ap ^\pm J_u
\end{equation}
\end{lemma}

\begin{proof}
  Differentiate equation~\eqref{e:define_avpm_semiinfi} with respect to $\tau$, 
  using equation~\eqref{e:diff_adjoint} and the fact that $\ad ^\tau _\tau$ is the identity map, we have
  \begin{equation} \begin{split}
    \dd {\av^\pm(\tau)} \tau =& \dd{}\tau \int_\tau^\infty \adttau \ap^-(t) J_u(t) dt - \dd{}\tau \int_0^\tau \adttau \ap^+(t) J_u(t) dt \, \\
    =& - \ad _\tau ^\tau \ap ^- (\tau) J_u(\tau)
    +  \int_\tau^\infty - f_u^T(\tau) \adttau \ap^-(t) J_u(t) dt\\
    &- \ad _\tau ^\tau \ap ^+ (\tau) J_u(\tau)
    - \int_0^\tau -f_u^T(\tau) \adttau \ap^+(t) J_u(t) dt \\
    =& -\ap^\pm(\tau) J_u(\tau) - f_u^T(\tau)\left[ \int_\tau^\infty \adttau \ap^-(t) J_u(t) dt - \int_0^\tau \adttau \ap^+(t) J_u(t) dt\right]\\
    =&  -\ap^\pm(\tau) J_u(\tau) - f_u^T(\tau) \av^\pm(\tau) \,,
  \end{split}
  \end{equation}
\end{proof}

\begin{lemma} \label{l:dv0 dt}
  $\av^0$ defined in equation (\ref{e:define_av0}) solves the inhomogeneous adjoint equation:
  \begin{equation}
    \dd {\av^0}\tau + f_u^T \av^0 = - \ap^0 J_u
  \end{equation}
\end{lemma}

\begin{proof}
  By lemma~\ref{l:adjoint_pair_operators_products} and that $f$ is a neutral tangent CLV, $\ip{\ay, f}$ is a constant.
  Now differentiate with respect to $\tau$,
  using the fact that $d\tilde{J}/d\tau = dJ/d\tau = \ip{J_u, f}$,  we have 
  \begin{equation} \begin{split}
    \dd {\av^0(\tau)} \tau 
    &= - \dd{}\tau \frac{ \tilde{J}(\tau) \ay(\tau) }{\ip{\ay(\tau),f(\tau)}} 
    = - \frac{1}{\ip{\ay(\tau),f(\tau)}} \dd{}\tau { \tilde{J}(\tau) \ay(\tau) } \\
    &= - \frac{1}{\ip{\ay,f}} \left(\dd{\tilde{J}}\tau \ay + \tilde{J} \dd \ay \tau\right) 
    = - \frac{1}{\ip{\ay,f}} \left(\ip{J_u, f} \ay - \tilde{J} f_u^T \ay \right) \,.
  \end{split} \end{equation}
  Using lemma~\ref{l:ap0_from_ay} and the definition of $\av^0$, we have:
  \begin{equation}
    \dd {\av^0(\tau)} \tau 
    = - \frac{\ip{J_u, f} }{\ip{\ay,f}} \ay + f_u^T \frac{\tilde{J} \ay }{\ip{\ay,f}}  
    = -\ap^0 J_u - f_u^T \av^0 \,
  \end{equation}
\end{proof}

By lemma \ref{l:dvpm dt}, lemma \ref{l:dv0 dt}, $\av = \av^\pm + \av^0$, and that $\ap^0 + \ap ^\pm = I$,
we have shown that $\av$ has the first property of adjoint shadowing directions.
\begin{proposition}
  $\av$ defined in equation~\eqref{e:define_av_semiinfi} solves the inhomogeneous adjoint equation:  
  \begin{equation}
    \dd {\av}\tau + f_u^T \av = - J_u \,.
  \end{equation}
\end{proposition}

\subsection{Proving the second property}

\begin{proposition}
  $\av$ defined in equation~\eqref{e:define_av_semiinfi} has zero component in the unstable adjoint subspace at $\tau=0$.
\end{proposition}

\begin{proof}
We can substitute $\tau=0$ into equation~\eqref{e:define_av_semiinfi} to get
\begin{equation} 
  \av (\tau=0)  =  \int_0^\infty \adttau \ap^-(t) J_u(t) dt + \av^0 (\tau)\,.
\end{equation}
By lemma~\ref{l:aDP=aPD} and definition of $\av^0$ in equation~\eqref{e:define_av0}, we have
\begin{equation} 
  \av (0)  =  \ap^-(0) \int_0^\infty \ad_t^0  J_u(t) dt - \frac{ \tilde{J} }{\ip{f,f}} \ap^0 f \,.
\end{equation}
Here the first term is in the stable adjoint subspace, and the second is in the neutral subspace,
hence $\av(0)$ has zero unstable adjoint component.
\end{proof}

\subsection{Proving the third property}

We prove the boundedness of $\av$ by proving the boundedness separately on $\av^\pm$ and $\av^0$.

\begin{lemma}\label{l:bound_avpm}
  $\av^\pm(\tau)$ defined in equation~\eqref{e:define_avpm_semiinfi} is bounded by a constant independent of $\tau$.
\end{lemma}

\begin{proof}
  Since $J_u$ is bounded, 
  by lemma \ref{l:ap_projects_to_adjoint_CLV} and uniform hyperbolicity, we have,
  \begin{equation} \begin{split}
    \|\av^\pm(\tau)\| =& \left \| \int_\tau^\infty \adttau \ap^-(t) J_u(t) dt - \int_0^\tau \adttau \ap^+(t) J_u(t) dt \right \| \\
    \le & \int_\tau^\infty \left \| \adttau \ap^-(t) J_u(t)\right \| dt 
    + \int_0^\tau \left \| \adttau \ap^+(t) J_u(t)\right \| dt\\
    \le &\int_\tau^\infty e^{-\lambda(t-\tau)} \left \| J_u(t)\right \| dt 
    + \int_0^\tau e^{\lambda(t-\tau)} \left \|J_u(t)\right \| dt\\
    \le & \left \|J_u\right\|_{L^\infty} 2 \int_0^\infty e^{-\lambda(t)dt} \, 
    = \frac 2\lambda \left \|J_u\right\|_{L^\infty} ,
  \end{split} \end{equation}
  where $\lambda$ is from the definition of uniform hyperbolicity.
\end{proof}

\begin{lemma}\label{l:bound_av0}
  $\av^0(\tau)$ defined in equation~\eqref{e:define_av0} is bounded by a constant independent of $\tau$.
\end{lemma}

\begin{proof}
  By lemma~\ref{l:adjoint_pair_operators_products} and that $f$ is a neutral tangent CLV,
  $\ip{\ay, f}$ is a constant.
  By lemma~\ref{l:ap_projects_to_adjoint_CLV} and that $\ay$ is the neutral adjoint CLV with zero exponent, 
  $\ay$ is bounded by a constant independent of $T$.
  Moreover, a continuous $\tilde{J}$ is bounded on a compact attractor,
  hence $\av^0=\tilde{J}\ay/\ip{\ay, f}$ is bounded.
\end{proof}

By lemma~\ref{l:bound_avpm}, lemma \ref{l:bound_av0} and that $\av = \av^\pm + \av^0$, 
we have shown that $\av$ has the third property of adjoint shadowing directions.
\begin{proposition}
  $\av(\tau)$ defined in equation~\eqref{e:define_av_semiinfi} is bounded by a constant independent of $\tau$.
\end{proposition}

\subsection{Proving the fourth property}

\begin{proposition}
  $\ip{\av, f}_{avg} = 0$.
\end{proposition}

\begin{proof}

  We prove by first showing that $\av^\pm(\tau)$ is perpendicular to $f(\tau)$ for any $\tau$, and then $\ip{\av^0, f}_{avg} = 0$.
  By that $\ap^-(t)$ is a projection operator so hence $\ap^-(t)\ap^-(t)=\ap^-(t)$,
  and by lemma~\ref{l:aDP=aPD},
  equation~\eqref{e:define_avpm_semiinfi} can be written as:
  \begin{equation} \begin{split}
    \av^\pm(\tau) &=  \int_\tau^\infty \adttau \ap^-(t)\ap^-(t)  J_u(t) dt -  \int_0^\tau \adttau \ap^+(t) J_u(t) dt \\
                  &= \ap^-(\tau) \int_\tau^\infty \adttau  \ap^-(t) J_u(t) dt - \ap^+(\tau) \int_0^\tau \adttau  J_u(t) dt \,,
  \end{split} \end{equation}
  where the projection  $\ap^-(t)$ inside the first integration is for the convergence of the integration.
  Since  $f$ is the direction $P^0$ projects to, by lemma~\ref{l:aP_P_orthogonal}, we see $\av^\pm(\tau)$ is perpendicular to $f(\tau)$.
  For $\av^0$, using the definition of $\tilde{J}$, we have
  \begin{equation}
    \ip{\av^0, f}_{avg}
    = \ip{ - \frac{ \tilde{J} }{\ip{\ay,f}} \ay, f }_{avg}
    =  - \tilde{J}_{avg}
    = 0
  \end{equation}
\end{proof}

\subsection{Proving uniqueness}

\begin{proposition}
  On a trajectory on the attractor, the adjoint shadowing direction is unique.
\end{proposition}

\begin{proof}
  We have shown that $\av$ is an adjoint shadowing direction.
  Assume that there is another $\av'$ satisfying definition~\ref{def:adjoint shadowing direction},
  then by the first property, $\av'$ solves the same inhomogeneous adjoint equation~\eqref{e:av solve inhomo}.
  Hence $\av' -\av$ solves the homogeneous adjoint equation~\eqref{e:homogeneous adjoint}.

  Now by the second property, $\av'(0)$ also has zero unstable adjoint components,
  hence the unstable adjoint components in $\av'-\av$ at $\tau=0$, and hence for all $\tau\ge 0$, is zero.

  Moreover, if $\av' -\av$ has a stable adjoint component at $\tau=0$, this difference will grow exponentially as $\tau$ grows larger.
  In order for both $\av$ and $\av'$ to have the boundedness required by the third property, the stable component in $\av'-\av$ must be zero.

  We still need to show the neutral adjoint component in $\av'-\av$ is zero.
  Assume $\av'-\av = C\ay$ for some constant $C$.
  By hyperbolicity, $\ay$ is at least angle $\alpha$ away from the non-neutral adjoint subspace, which is also the subspace perpendicular to $f$.
  Hence $\ip{\ay, f}\ne 0$ at any time.
  Moreover, by lemma~\ref{l:adjoint_pair_operators_products},  $\ip{\ay, f}$ is a constant.
  Hence $\avg{\ip{\ay,f}} \ne 0$.
  In order for both $\av$ and $\av'$ to have the fourth property, we must have $C=0$.
  We have thus concluded $\av=\av'$.
\end{proof}

Intuitively, the first property demands possible candidates for adjoint shadowing directions be in the affine space
of a particular inhomogeneous adjoint solution plus a linear combination of adjoint CLVs. 
Then the second, third, and fourth property prescribe coefficients for the adjoint unstable, stable, and neutral adjoint subspaces.
We have thus uniquely determined our adjoint shadowing direction.

\section{Adjoint shadowing directions for hyperbolic diffeomorphisms} \label{s:diffeo}

In this section we discuss adjoint shadowing directions of discrete dynamical systems given by hyperbolic diffeomorphisms.
Because the absence of the neutral subspace, the discussion about diffeomorphisms will be easier than flows.
However, managing subscripts here calls for more caution.
This structure of this section will be a miniature of the entire paper,
and notations here are the same as in the case of flows unless otherwise noted.

\subsection{Preparations}

The homogeneous tangent diffeomorphism is: 
\begin{equation}\label{e:homo_tangent_diffeo}
  w_{i+1} = f_{ui} w_i \,.
\end{equation}
where the second subscript of $f_{ui}$ indicate where the partial derivative is evaluated, that is,
$f_{ui}:=$ $\partial f/ \partial u (u_i,s)$.
The inhomogeneous tangent diffeomorphisms have an additional right hand term independent of $v_i$, 
and the particular inhomogeneous tangent diffeomorphism we will be using is:
\begin{equation}\label{e:inhomo_tangent_diffeo}
  v_{i+1} = f_{ui} v_i + f_{si} \,.
\end{equation} 
where $f_{si}:=$ $\partial f/ \partial s (u_i,s)$.
The propagation operator $D_l^i$ is defined as the matrix that maps a homogeneous tangent solution at step $l$ to step $i$.
More specifically,
\begin{equation} \begin{split}
  D_l^i := 
  \begin{cases}
    I, \quad \textnormal{ when } i=l; \\
    f_{u,i-1} f_{u,i-2} \cdots f_{u,l+1}f_{u,l}, \quad \textnormal{ when } i>l; \\
    (f^{-1})_{u,i+1}\cdots (f^{-1})_{u,l}
      = f^{-1}_{u,i} f^{-1}_{u,i+1} \cdots f^{-1}_{u,l-2} f^{-1}_{u,l-1}, \quad \textnormal{ when }i<l. 
  \end{cases}
\end{split} \end{equation}
A tangent CLV with exponent $\lambda$ is a homogeneous tangent solution $\{\zeta_i\}^\infty_{i=0}$ such that there is constant $C$,
for any integer $i_1, i_2$, $\| \zeta_{i_2} \| \le C e^{\lambda(i_2-i_1)}\| \zeta_{i_1} \|$. 
The uniform hyperbolicity is defined as follows.

\begin{definition}[uniform hyperbolic diffeomorphisms]
  In this paper, a diffeomorphism is said to be uniform-hyperbolic if 
  there are $C\in(0,\infty)$ and $\lambda>0$, 
  such that for all $u$ on the attractor, there is a splitting of the tangent space $T_u=\R^m$ into
  $T_u = V^+(u)\oplus V^-(u)$, 
  and on the trajectory passing $u$ at step $0$, we have:
  \begin{enumerate}
    \item for any $v\in V^+(u)$ and $i\le -1$, $\|D_0^i v\|\le C e^{ -\lambda |i|} \|v\|$;
    \item for any $v\in V^-(u)$ and $i\ge 1$,  $\|D_0^i v\|\le C e^{ -\lambda |i|} \|v\|$.
  \end{enumerate}
\end{definition}

Different from the flow case, there is no neutral subspace.
Similar to the flow case, there are in total $m$ different tangent CLVs,
and with uniform hyperbolicity, we can show that the absolute value of all LEs are greater than $\lambda$.
We can also show that the angle between any CLV and the span of the rest is larger than some $\alpha>0$.
Hence the CLVs form a basis of $\R^m$.
Similar to definition~\ref{def:tangent P}, we can define three projection operators, $P^j_i, P^+_i, P^-_i$, 
projecting to the $j$-th CLV, the unstable tangent subspace, and the stable tangent subspace at step $i$.
Similar to lemma~\ref{l:DP=PD}, we can show that $D_l^i P_l = P_i D_l^i$.
Similar to lemma~\ref{l:Ca}, we can show there is $C_\alpha>0$ such that $\|Pv\|\le C_\alpha \|v\|$ for all three projection operators.

The homogeneous adjoint diffeomorphism is:
\begin{equation} \label{e:homo_adjoint_diffeo}
  \aw_{l} = f_{ul}^T \aw_{l+1} \,,
\end{equation}
where $\cdot^T$ is the matrix transpose.
The particular inhomogeneous adjoint diffeomorphism we will be using is:
\begin{equation} \label{e:inhomo_adjoint_diffeo}
  \av_{l} = f_{ul}^T \av_{l+1} + J_{ul} \,,
\end{equation}
where  $J_{ui} :=$ $ \partial J/ \partial u (u_i,s)$.
The adjoint propagation operator $\ad_i^l$ is defined as the matrix that maps a homogeneous adjoint solution at step $i$ to step $l$:
\begin{equation} \begin{split}
  \ad_i^l := 
  \begin{cases}
    I, \quad \textnormal{ when } i=l; \\
    f_{u,l}^T f_{u,l+1}^T \cdots f_{u,i-2}^T f_{u,i-1}^T, \quad \textnormal{ when } i>l; \\
    f^{-T}_{u,l-1} f^{-T}_{u,l-2} \cdots f^{-T}_{u,i+1} f^{-T}_{u,i}, \quad \textnormal{ when }i<l. 
  \end{cases}
\end{split} \end{equation}
A direction computation shows that similar to lemma~\ref{l:adjoint_pair_operators_products}, $(D_l^i)^T = \ad_i^l$.

An adjoint CLV with exponent $\lambda$ is a homogeneous adjoint solution $\{\az_i\}_{i=0}^\infty$ such that there is constant $C$,
for any integer $i_1, i_2$, $\| \az_{i_1} \| \le C e^{\lambda(i_2-i_1)}\| \az_{i_2} \|$. 
We define adjoint projection operators as the transpose of tangent projection operators, that is, $\ap := P^T$.
Similar to lemma~\ref{l:ap_projects_to_adjoint_CLV}, we can show that $\ap^j$ projects to the adjoint CLV with exponent $\lambda_j$, 
which is also the exponent of the $j$-th tangent CLV. 
Similar to lemma~\ref{l:bound for ap}, we can show there is $C_\alpha>0$ such that $\|\ap v\|\le C_\alpha \|v\|$ for all three adjoint projection operators.

\subsection{A candidate formula}

While proving the convergence of the Least Squares Shadowing method for hyperbolic diffeomorphisms \cite{wang2014convergence},
Wang also proved that, under the same assumption of theorem \ref{thm:main theorem for diffeomorphisms}, we have 
\begin{equation} \begin{split}\label{e:djds_lss_diffeo}
  \dd {\avg{J}} {s} &= \lim_{N\rightarrow \infty} \frac 1N \summation \left(\ip{J_{ui}, v_{Ni}} + J_{si} \right) \,, \textnormal{ where } \\
  v_{Ni} &= \sum_{l=0}^{i-1}D^i_{l+1}P_{l+1}^-f_{sl} - \sum_{l=i}^{N-1} D_{l+1}^i P_{l+1}^+ f_{sl} \,.
\end{split} \end{equation}
Using the fact that $f_{ui} D^i_{l+1} = D^{i+1}_{l+1}$, 
we can check that $v_{Ni}$ solves the inhomogeneous tangent equation~\eqref{e:inhomo_tangent_diffeo}.
Pilyugin~\cite{Pilyugin_shadow_linear_formula,Pilyugin1999book} defined a sequence $\{v_i\}_{i=0}^\infty$, with
$v_i:=\lim_{N\rightarrow \infty}v_{Ni}$, which we here call the shadowing direction.
Pilyugin also shows the shadowing direction is bounded for hyperbolic diffeomorphisms.

Similar to section~\ref{s:candidate formula}, we derive a $\av$ satisfying the adjoint sensitivity formula in equation~\eqref{e:djds_adjoint_diffeo}.
To achieve this, we separate the dependency on $f_s$ in equation~\eqref{e:djds_lss_diffeo} as follows.
\begin{equation} \begin{split}
  \summation \ip{J_{ui}, v_{Ni}}  
  &= \summation \ip{J_{ui}, \sum_{l=0}^{i-1}D^i_{l+1}P_{l+1}^-f_{sl} - \sum_{l=i}^{N-1} D_{l+1}^i P_{l+1}^+ f_{sl}}\\
  &= \summation \sum_{l=0}^{i-1} \ip{J_{ui}, D^i_{l+1}P_{l+1}^-f_{sl}} -\summation \sum_{l=i}^{N-1} \ip{J_{ui}, D_{l+1}^i P_{l+1}^+ f_{sl}}\\
  &= \sum_{l=0}^{N-1} \sum_{i=l+1}^{N-1} \ip{J_{ui}, P_{i}^- D^i_{l+1}f_{sl}} - \sum_{l=0}^{N-1} \sum_{i=0}^{l}\ip{J_{ui},  P_{i}^+ D_{l+1}^i f_{sl}}\\
  &= \sum_{l=0}^{N-1} \sum_{i=l+1}^{N-1} \ip{\ad_i^{l+1} \ap_{i}^-J_{ui},  f_{sl}} - \sum_{l=0}^{N-1} \sum_{i=0}^{l}\ip{ \ad^{l+1}_i \ap_{i}^+J_{ui}, f_{sl}}\\
  &= \sum_{l=0}^{N-1}  \ip{ \sum_{i=l+1}^{N-1} \ad_i^{l+1} \ap_{i}^-J_{ui}  -  \sum_{i=0}^{l} \ad^{l+1}_i \ap_{i}^+J_{ui}, f_{sl}} \,.
 \end{split} \end{equation}
We define a sequence $\{\av_{Nl}\}_{l=0}^{N}$ as
\begin{equation}
  \av_{Nl} = \sum_{i=l}^{N-1} \ad_i^{l} \ap_{i}^-J_{ui}  -  \sum_{i=0}^{l-1} \ad^{l}_i \ap_{i}^+J_{ui} \,.
\end{equation}
In our summation symbols, when the lower bound is strictly larger than the upper bound, that summation is zero.
In particular, when $l=0$, $\av_{N0} = \sum_{i=0}^{N-1} \ad_i^{0} \ap_{i}^-J_{ui}$.
Hence we have proved that
\begin{equation} \label{e:sum equivalence diffeo}
  \summation \ip{J_{ui}, v_{Ni}} = \sum_{l=0}^{N-1} \ip{\av_{N,l+1}, f_{sl}} \,.
\end{equation}
Notice that in the summation, the subscripts of $v_{Ni}$ runs from $i=0$ to $N-1$, whereas the subscripts of $\av_{Nl}$ runs from $l=1$ to $N$.

Similar to section~\ref{s:extend to semi infinite}, we can define a semi-infinite sequence $\{\av_{l}\}_{l=0}^{\infty}$ as
\begin{equation}
  \av_l := \lim_{N\rightarrow\infty}\av_{Nl} = \sum_{i=l}^{\infty} \ad_i^{l} \ap_{i}^-J_{ui}  -  \sum_{i=0}^{l-1} \ad^{l}_i \ap_{i}^+J_{ui} \,.
\end{equation}
using the same method as in section~\ref{s:extend to semi infinite}, we can show that above limit exists, 
and we have the adjoint sensitivity formula in equation~\eqref{e:djds_adjoint_diffeo}.
$\{\av_{l}\}_{l=0}^{\infty}$ is our candidate for adjoint shadowing direction of hyperbolic diffeomorphisms.

\subsection{Existence and uniqueness}

Using the fact that $f_{ul}^T \ad^{l+1}_i = \ad^l_i$, we can show that $\{\av_{l}\}_{l=0}^{\infty}$  solves the inhomogeneous adjoint equation:
\begin{equation}\begin{split}
  \av_l - f_{ul}^T \av_{l+1}
  &= \sum_{i=l}^{\infty} \ad_i^{l} \ap_{i}^-J_{ui}  -  \sum_{i=0}^{l-1} \ad^{l}_i \ap_{i}^+J_{ui}
  - \sum_{i=l+1}^{\infty} \ad_i^{l} \ap_{i}^-J_{ui}  +  \sum_{i=0}^{l} \ad^{l}_i \ap_{i}^+J_{ui}\\
  &= \ad_l^{l} \ap_{l}^-J_{ul} +  \ad^{l}_l \ap_{l}^+J_{ul}
  =\ap_{l}^-J_{ul} + \ap_{l}^+J_{ul} = J_{ul}
\end{split}\end{equation}

By observing the formula of $\av_0$ we find it has no component in the adjoint unstable directions, thus satisfying the second property.
Notice that in the summation in equation~\eqref{e:sum equivalence diffeo} does not involve $\av_{N0}$
and the average in equation~\eqref{e:djds_adjoint_diffeo} does not involve $\av_0$.
However, we still start the sequence $\av$ from $\av_0$,
not only because adding $ \ip{\av_{0}, f_{s,-1}}$ to the average in equation~\eqref{e:djds_adjoint_diffeo} does not change its limit,
but also because we can impose the second property neatly by using $\av_0$, which is essential for uniqueness of adjoint shadowing directions.

Then we show $\av$ is bounded.
Due to hyperbolic assumption,
\begin{equation} \begin{split}
  \left\| \ad_i^{l} \ap_{i}^-J_{ui} \right\| &\le C e^{-\lambda|l-i|} \left \|\ap_{i}^-J_{ui} \right\|
  \le C C_\alpha e^{-\lambda|l-i|} \left \|J_{ui} \right\|
  \,, \quad \textnormal{ when } i>l; \\
  \left\| \ad^{l}_i \ap_{i}^+J_{ui} \right\| &\le C e^{-\lambda|l-i|} \left \|\ap_{i}^+J_{ui} \right\| 
  \le C C_\alpha e^{-\lambda|l-i|} \left \|J_{ui} \right\|
  \,, \quad \textnormal{ when } i<l\,.
\end{split} \end{equation}
Here $C, \lambda$ are from the definition of uniform hyperbolicity, $C_\alpha$ is from an analogy of lemma~\ref{l:bound for ap}.
Using that a continuous $J_u$ is bounded on a compact attractor, we have
\begin{equation}
  \| \av_l \| \le  C C_\alpha \left \| J_u \right \|_{L^\infty}
  \left( \sum_{i=l}^{\infty} e^{-\lambda|l-i|} +  \sum_{i=0}^{l-1} e^{-\lambda|l-i|} \right) 
  = \frac{ 2 C C_\alpha \left \| J_u \right \|_{L^\infty}} {1- e^{-\lambda}}
  \,.
\end{equation}

Finally, we show uniqueness.
Under the constraint of the first property, our candidate sequence, say $\av'$, is determined once we give the initial condition at $0$-th step.
The second property tells us what initial condition we should have on the unstable subspace.
Now we have only freedom to decide the initial condition on the stable subspace.
Assume $\av'_0$ differs from $\av_0$ on the stable adjoint subspace, this difference would grow exponentially as we step forward.
In order for both $\av$ and $\av'$ to have the boundedness required by the third property,
$\av$ and $\av'$ must be identical on the stable subspace.

\appendix

\section{Properties of the tangent flow} \label{app:properties_of_tangent_flow}
This appendix proves properties of the tangent flow and its corresponding projection operators.

\begin{lemma} \label{l:matrix_form_of_P}
  $P(t)=Z(t)DZ(t)^{-1}$, where $D$ is a constant diagonal matrix with entries being one or zero, 
  and $Z(t)$ is the square matrix whose column vectors are CLVs at time $t$: $Z(t) = [\zeta_1(t),\cdots, \zeta_n(t)]$. 
  Here the projection operator $P$ can be either $P^j, P^+, P^-, P^0$, or $P^\pm$, 
  and the corresponding diagonal matrices are denoted by $D^j, D^+, D^-, D^0$, or $D^\pm$, respectively.
\end{lemma}

\begin{proof}
  We first prove the case for $P^j$.
  For any vector $v\in \R^m$, at time $t$, we decompose $v$ onto basis $\{\zeta_j(t)\}_{j=1}^m$.
  This decomposition can be written in matrix form:
  \begin{equation}
    Z(t) \,a =  [\zeta_1(t),\cdots, \zeta_n(t)]\,a = v \,,
  \end{equation}
  where $a = [a_1,\cdots,a_m]^T$ is the coordinate.
  Since the CLVs form a basis of $\R^m$, $Z(t)$ is invertible.
  We define $D^j$ as the matrix whose only non-zero entry is the $j$-th entry on the diagonal, and the value is 1.
  Now using that $P^j(t)v = a_j\zeta_j(t)$ by definition,
  \begin{equation}
    Z(t)D^jZ(t)^{-1}v = Z(t)D^j\,a = Z(t) [0,\cdots,0,a_j,0,\cdots,0]^T=a_j\zeta_j(t) = P^j(t)v\,.
  \end{equation}
  Above equivalence holds for any $v$, hence $P^j(t)=Z(t)D^jZ(t)^{-1}$.
  The diagonal matrices for other projection operators are formed by summing corresponding $D^j$'s.
\end{proof}

\begin{lemma} \label{l:DP=PD}
  $\dtaut P(\tau) = P(t) \dtaut$, where the projection operator $P$ can be either $P^j, P^+$, $P^-$, $P^0$, or $P^\pm$.
\end{lemma}

\begin{proof}
  We first prove the case for $P^j$.
  For any $v\in\R^m$, at $\tau$, we decompose it onto basis $\{\zeta_j(t)\}_{j=1}^m$. 
  By definition, $P^j(\tau)v = a_j \zeta_j(\tau)$. Since $\zeta_j$ is a homogeneous tangent solution, 
  \begin{equation}
    \dtaut P^j(\tau)v =  a_j \dtaut \zeta_j(\tau) = a_j \zeta_j(t) \,.
  \end{equation}
  On the other hand, 
  \begin{equation}
    P^j(t) \dtaut v 
    = P^j(t) \sum_{k=1}^{n} a_k \dtaut  \zeta_k (\tau)  
    = P^j(t) \sum_{k=1}^{n} a_k \zeta_k (t)
    = a_j\zeta_j(t) \,.
  \end{equation}
  Both equivalences holds for any $v$, hence $\dtaut P^j(\tau) = P^j(t) \dtaut$.
  Other cases follow from the fact that other projection operators are summations of several $P^j$.
\end{proof}

\begin{lemma}\label{l:Ca}
  Assume we can find $\alpha>0$ such that at any time $t$, 
  the angle between any $\zeta_j(t)$ and the span of the rest CLVs, 
  $span\left\{\{\zeta_k(t)\}_{k\ne j} \right\}$, is larger than $\alpha$.
  Then there exists $C_\alpha$ such that for any $v\in R^m$, at any $t$, 
  and for all tangent projection operators $P$ (either $P^j, P^+, P^-, P^0$, or $P^\pm$),
  we have $\|P(t)v\|\le C_\alpha\|v\|$.
\end{lemma}

\begin{proof}
  For any projection $P$,
  as shown in figure~\ref{f:tangent project}, we denote $v_2:= v - P v$ and $\phi$ the angle between $P v$ and $v_2$.
  By our assumption, $\alpha$ lower-bounds the angle between any two subspaces spanned by different sets of CLVs,
  hence $\alpha \le \phi \le \pi-\alpha$, which further indicates $|\cos \phi| \le \cos \alpha$.
  By the law of cosine, we have
  \begin{equation}\begin{split}
    \|v\|^2 &= \|P v\|^2 + \|v_2\|^2 - 2 \|P v\|\|v_2\| \cos(\pi-\phi) \\
            &= (1-\cos^2 \phi)\|P v\|^2 + (\cos\phi \|P v\| + \|v_2\|)^2 \\
            &\ge (1-\cos^2 \alpha)\|P v\|^2 \,.
  \end{split}\end{equation}
  Define $C_\alpha = (1-\cos^2 \alpha)^{-1/2}$, we have $\|Pv\| \le C_\alpha \|v\|$.

  \begin{figure} \centering
    \includegraphics[width=0.45\textwidth]{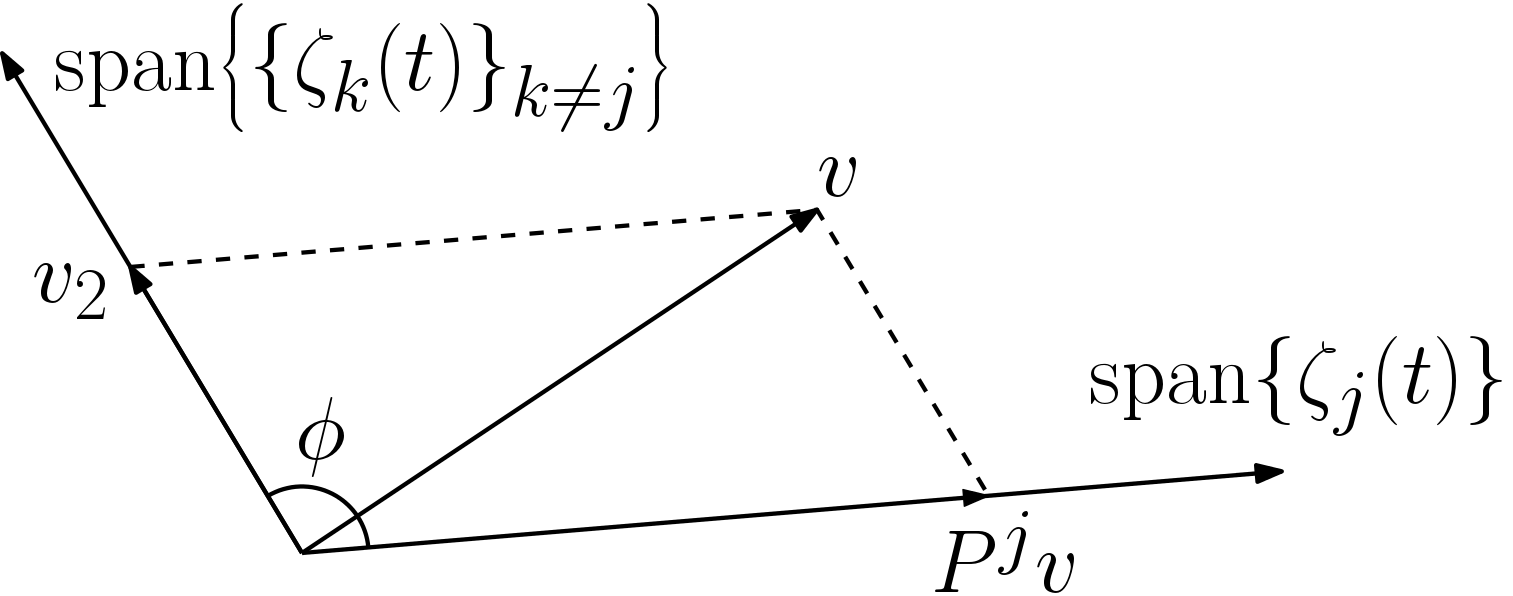}
    \caption{Tangent projection operators.}
    \label{f:tangent project}
  \end{figure}
\end{proof}

\section{Properties of the adjoint flow} \label{app:properties_of_adjoint_flow}
This appendix proves properties of the adjoint flow and its corresponding projection operators.

\begin{lemma}\label{l:adjoint_pair_operators_products}
  For any $t_1, t_2\in \R$, any $w_1,\aw_2 \in \R^m$, $\ip{\overline{D}^{t_1}_{t_2}\aw_2,w_1} = \ip{\aw_2, D^{t_2}_{t_1}w_1}$.
  In other words, for any homogeneous tangent solution $w(t)$ and homogeneous adjoint solution $\aw(t)$, $\ip{\aw(t),w(t)}$ is a constant.
\end{lemma}

\begin{proof}
  Without loss of generality, we assume $t_1<t_2$.
  On time span $[t_1,t_2]$, we solve the homogeneous tangent $w(t)$ with initial condition $w(t_1) = w_1$, 
  and the homogeneous adjoint $\aw(t)$ with terminal condition $\aw(t_2) = \aw_2$.
  The equality we want to prove can now be written as $\ip{\aw(t_1),w(t_1)} = \ip{\aw(t_2), w(t_2)}$.
  To see this, we compute the time derivative:
  \begin{equation} \begin{split}
    \frac{d}{dt} \ip{\aw(t),w(t)} 
    &= \ip{\dd{\aw(t)}t,w(t)} +\ip{\aw(t),\dd{w(t)}t} \\
    &= \ip{-f_u^T \aw(t),w(t)} +\ip{\aw(t),f_u w(t)} = 0 \,.
  \end{split} \end{equation}
  Hence $\ip{\aw(t),w(t)}$ is a constant.
\end{proof}

\begin{lemma} \label{l:aP_P_orthogonal}
  At any time $t$, for any $i\ne j$, and any $v,w\in \R^m$, $\ip{\ap^i(t) w, P^j(t) v} = 0$.
\end{lemma}

\begin{proof}
  By lemma \ref{l:adjoint_pair_operators_products},
  \begin{equation}
    \ip{\ap^i(t) w, P^j(t) v} = \ip{ w, P^i(t) P^j(t) v} 
  \end{equation}
  Using lemma \ref{l:matrix_form_of_P} to write $P^i$ and $P^j$ in matrix form, we have 
  \begin{equation} 
    \ip{\ap^i(t) w, P^j(t) v} 
    = \ip{ w,  Z(t)D^i D^jZ(t)^{-1} v}  =0 \,,
  \end{equation}
  where we used $D^i D^j=0$, since $D^i$ and $D^j$ are two diagonal matrices with non-zero entries at different locations.
\end{proof}

\begin{lemma} \label{l:aDP=aPD}
  For any $t_1, t_2\in\R$, $\adto \ap(t_2) = \ap(t_1)\adto$. 
  The adjoint projection operator $\ap$ can be either $\ap^j, \ap^+, \ap^-, \ap^0$, or $\ap^\pm$.
\end{lemma}
\begin{proof}
  For any $w, v\in\R^m$, 
  using lemma \ref{l:DP=PD}, lemma \ref{l:adjoint_pair_operators_products}, and the definition of $\ap$, we have
  \begin{equation}
    \ip{\adto \ap(t_2)w, v} = \ip{w, P(t_2)\dtot v} = \ip{w, \dtot P(t_1) v}  = \ip{\ap(t_1)\adto w, v}  \,.
  \end{equation}
  Hence $\ip{\left(\adto \ap(t_2) - \ap(t_1)\adto \right) w, v} = 0$ for any $w, v\in\R^m$.
  This implies the lemma. 
\end{proof}

\begin{lemma}\label{l:bound for ap}
  Under the same assumptions and for the same $C_\alpha$ of lemma~\ref{l:Ca}, 
  we have the same conclusion for adjoint projection operators.
  Than is, for any $v\in R^m$, at any $t$, 
  and for all adjoint projection operators $\ap$ (either $\ap^j, \ap^+, \ap^-, \ap^0$, or $\ap^\pm$),
  we have $\|\ap(t) v\|\le C_\alpha\|v\|$.
\end{lemma}

\begin{proof}
  For any $v, w\in \R^m$, we have
  \begin{equation} \begin{split}
    |\ip{w, \ap v}| = |\ip{Pw, v}| \le \|Pw\| \|v\| \le C_\alpha \|w\| \|v\|\,,
  \end{split} \end{equation}
  where we used lemma~\ref{l:Ca}.
  Now let $w =  \ap v$, we have
  \begin{equation} \begin{split}
    \| \ap v \|^2 \le C_\alpha \|\ap v\| \|v\|\,.
  \end{split} \end{equation}
  The lemma is proved by canceling $\|\ap v\|$ on each side.
\end{proof}

\begin{proposition} \label{l:ap_projects_to_adjoint_CLV}
  For any $j$, there is a constant $\ac _j$ such that, for any $t_1,t_2\in \R$, $w_2\in \R^m$,
  $\| \adto \ap ^j(t_2) w_2\| \le \ac _j e^{\lambda_j(t_2-t_1)} \|\ap ^j(t_2) w_2\|$.
  In other words, $\ap^j$ projects to the $j$-th adjoint CLV.
\end{proposition}

\begin{proof}
  For any $v_1\in R^m$, 
  \begin{equation}
    \ip{\adto \ap ^j(t_2) w_2, v_1} = \ip{w_2, \dtot P^j (t_1) v_1} \, .
  \end{equation}
  Letting $ v_1 = \adto \ap ^j(t_2) w_2$, applying Cauchy-Schwarz inequality,
  and recalling $\|D_{t_1}^{t_2} P^j(t_1)v\| \le C_j e^{\lambda_j(t_2-t_1)} \| P^j(t_1)v\|$ for some $C_j$, we have 
  \begin{equation} \begin{split} 
    \| \adto \ap ^j(t_2) w_2\|^2 =& \ip{w_2, \dtot P^j (t_1 )\adto \ap ^j(t_2) w_2} \\
    \le & \| w_2\| \, C_j e^{\lambda_j(t_2-t_1)}\|P^j (t_1 )\adto \ap ^j(t_2) w_2\| \\
    \le & \| w_2\| \, C_j e^{\lambda_j(t_2-t_1)} C_\alpha \| \adto  \ap ^j(t_2) w_2 \| \, ,
  \end{split} \end{equation}
  where $C_\alpha$ is the constant in lemma~\ref{l:bound for ap}.
  Cancel $\| \adto  \ap ^j(t_2) w_2 \|$ from both side of the inequality and set $\ac_j = C_j C_\alpha$, we get
  \begin{equation} \begin{split} 
    \| \adto \ap ^j(t_2) w_2\|
    \le & \ac_j e^{\lambda_j(t_2-t_1)}  \| w_2\| \,  \, .
  \end{split} \end{equation}
  This inequality holds for any $w_2$, in particular, we can pass $w_2$ to $\ap ^j(t_2) w_2$,
  \begin{equation} \begin{split} 
    \| \adto \ap ^j(t_2) \ap ^j(t_2) w_2\|
    \le & \, \ac_j e^{\lambda_j(t_2-t_1)}  \| \ap ^j(t_2) w_2\| \,  \, .
  \end{split} \end{equation}
  The lemma follows from the fact that $\ap^j$ is a projection operator, hence $\ap^j\ap^jw_2$ = $\ap^j w_2$.
\end{proof}

\begin{lemma} \label{l:ap0_from_ay}
  For any $t\in\R$, provided $\ay(t)$ and $f(t)$, then for any $v\in\R^m$, 
  \begin{equation}
    \ap^0(t) v  = \frac{\ip{v,f(t)}}{\ip{\ay(t), f(t)}} \ay(t) \,.
  \end{equation}
\end{lemma}

\begin{proof}
  For any $v\in \R^m$, $\ap^0 v$ is along the direction of the neutral adjoint CLV, $\ay$, so we can assume that
  $\ap^0 v = x \ay $,
  and our goal is to solve for the unknown coefficient $x$.
  On the other hand, by lemma \ref{l:aP_P_orthogonal}, the non-neutral adjoint subspace is perpendicular to $f$, so we have
  \begin{equation}
    \ip{v,f(t)} - x \ip{\ay(t), f(t)} = \ip{(I-\ap^0) v, f } = \ip{\ap ^\pm v, f } = 0 \,.
  \end{equation}
  Then we can solve for the unknown $x$ as given in the lemma.
\end{proof}

\bibliographystyle{siamplain}
\bibliography{MyCollection}

\end{document}